\documentclass[reqno]{amsart}
\usepackage{amssymb}
\usepackage{setspace}
\usepackage{hyperref}
\newtheorem{theorem}{Theorem}[section]
\newtheorem{lem}[theorem]{Lemma}

\newtheorem{corollary}[theorem]{Corollary}

\newtheorem{defn}[theorem]{Definition}
\newtheorem{example}[theorem]{Example}

\newtheorem{remark}[theorem]{Remark}

\numberwithin{equation}{section}

\begin{document}
\title[]{A common generalization of Dickson polynomials, Fibonacci polynomials, and Lucas polynomials and applications}
\author{ Said Zriaa and Mohammed Mou\c{c}ouf}
\address{ Said Zriaa and Mohammed Mou\c{c}ouf,
University Chouaib Doukkali.
Department of Mathematics, Faculty of science
Eljadida, Morocco}
\email{saidzriaa1992@gmail.com}
\email{moucouf@hotmail.com}
\subjclass[2020]{Primary 11B39; secondary 11B83.}
\keywords{Dickson polynomials; Fibonacci numbers; Lucas numbers; Fibonacci polynomials; Lucas polynomials.}

\begin{abstract}
In this work, we define a more general family of polynomials in several variables satisfying a linear recurrence relation. Then we provide explicit formulas and determinantal expressions. Finally, we apply these results to recurrent polynomials of order $2$, we present several relations and interesting identities involving the Fibonacci polynomials of order $2$, the Lucas polynomials of order $2$, the classical Fibonacci polynomials, the classical Lucas polynomials, the Fibonacci numbers, the Lucas numbers, the Dickson polynomials of the first kind, and the Dickson polynomials of the second kind. Our results are a unified generalization of several works. Some well known results are special cases of ours.
\end{abstract}

\maketitle	
\section{Introduction}
The classical Fibonacci polynomials $F_{n}(x)$ and the classical Lucas polynomials $L_{n}(x)$ are defined by
\begin{equation*}
F_{0}(x)=0,\,\ F_{1}(x)=1,\,\,\ F_{n+2}(x)=xF_{n+1}(x)+F_{n}(x), n\geq 0,
\end{equation*}
\begin{equation*}
L_{0}(x)=2,\,\ L_{1}(x)=x,\,\,\ L_{n+2}(x)=xL_{n+1}(x)+L_{n}(x), n\geq 0.
\end{equation*}
It is well known that the explicit expressions of the sequences $\{F_{n}(x)\}_{n\geq0}$ and $\{L_{n}(x)\}_{n\geq0}$ are given by the Binet formula,
\begin{equation*}
F_{n}(x)=\frac{1}{\alpha-\beta}(\alpha^{n}-\beta^{n}),\,\  \mbox{and} \,\,\ L_{n}(x)=\alpha^{n}+\beta^{n},\,\  \mbox{for all} \,\,\  n\geq 0,
\end{equation*}
where $\alpha=\frac{1}{2}(x+\sqrt{x^{4}+4})$ and $\beta=\frac{1}{2}(x-\sqrt{x^{4}+4})$.\\
\indent
These polynomials are of great importance in the study of various topics such as number theory, algebra, combinatorics, statistics, geometry, approximation theory, and other areas. It is not easy to describe versatile applications that rely on the classical Fibonacci polynomials and the classical Lucas polynomials. For specific references to some applications, the reader can consult for example~\cite{Pcatarino,Pcatarino2,Jlram,Ntuglu,Xye}. The classical Fibonacci and Lucas polynomials have always attracted the attention of several researchers. Therefore, many generalizations and research works have been considered in recent studies on the subject; some of them can be found in~\cite{Lchenxwang,Pericci,Ysoykan,Twang,Yyiwp}. Regardless of the fact that they are of great importance, to the best of our knowledge, no one has yet considered and studied the general case.\\
\indent
The theory of sequence polynomials can be applied to have powerful results on certain integer sequences including Fibonacci numbers, Lucas numbers, and a wide range other of sequences.\\
It is clear that when $x=1$, the classical Fibonacci polynomials turn into the well known Fibonacci numbers, $F_{n}(1)=F_{n}$. Also, it is clear that when $x=1$, the classical Lucas polynomials turn into the well known Lucas numbers, $L_{n}(1)=L_{n}$.\\
The Fibonacci and Lucas numbers are of intrinsic interest and have various fascinating properties. They too continue to amaze mathematicians with their splendid beauty, applicability, and ubiquity. They provide delightful opportunities to explore, experiment, conjecture, and problem-solve. The Fibonacci and Lucas numbers form a unifying thread intertwining geometry, analysis, trigonometry, and numerous areas of discrete mathematics such as combinatorics, linear algebra, number theory, and graph theory. For a deep and extensive survey of the theory and applications of the Fibonacci and Lucas numbers, we refer the reader to the book~\cite{Pell}. That research monograph contains not only a comprehensive treatise on this topic but contains in one location all currently known results concerning these numbers and their numerous applications and an extensive bibliography.
\\
\indent
In this paper, we study general polynomial sequences in several variables of high order. Further, we investigate these polynomials from different points of view, and we provide interesting explicit formulas and elegant determinantal expressions. Using these results, we obtain numerous interesting identities involving the Fibonacci polynomials order $2$ and the Lucas polynomials of order $2$. Consequently, several relationships between the classical Fibonacci and Lucas polynomials are presented. Also, relationships between the Dickson polynomials of the first and the second kind are provided. The results presented in this work can be used to recuperate, generalize, and develop various essential works on this important topic.
\section{Recurrent polynomials in several variables over unitary commutative ring}
Let $R$ be a commutative ring with identity. In the commutative ring $R[x_{1},x_{2},\cdots,x_{k}]$ of polynomials in the variables $x_{1},x_{2},\cdots,x_{k}$, we consider the sequence of polynomials $(P_{n}(x_{1},x_{2},\cdots,x_{k}))_{n\geq 0}$ satisfying the following general recurrence relation 
\begin{equation}\label{eq: standard}
P_{n+k}(x_{1},x_{2},\cdots,x_{k})=\sum_{i=1}^{k}c_{i}P_{n+k-i}(x_{1},x_{2},\cdots,x_{k}), \,\,\ n\geq 0.
\end{equation}
where for $i=1,2,\ldots,k$, we have $c_{i}=q_{i}(x_{1},x_{2},\cdots,x_{k})$ is a polynomial in $x_{1},x_{2},\cdots,x_{k}$.\\
For $j=0,1,2,\ldots,k-1$, we consider the elements $P^{(j)}_{n}(x_{1},x_{2},\cdots,x_{k})$ satisfying the recurrence relation~\eqref{eq: standard} with the initial conditions
$P^{(j)}_{i}(x_{1},x_{2},\cdots,x_{k})=\delta_{ij}$ for $i=0,1,2,\ldots,k-1$.\\
The polynomial
\begin{equation}\label{char}
g(X)=X^{k}-c_{1}X^{k-1}-c_{2}X^{k-2}-\cdots-c_{k}
\end{equation}
is called the characteristic polynomial of each elements solution to this recurrence relation.\\
In general, we call the polynomials solution to the equation~\eqref{eq: standard} the recurrent polynomials of order $k$, and it is clear that they can be written as
\begin{equation}\label{eq: g1}
P_{n}(x_{1},x_{2},\cdots,x_{k})=\sum_{i=0}^{k-1}P_{i}(x_{1},x_{2},\cdots,x_{k})P^{(i)}_{n}(x_{1},x_{2},\cdots,x_{k}), \,\ \mbox{for} \,\ n\geq 0. 
\end{equation}
\begin{example}
In the case where $c_{i}=(-1)^{i+1}$, the recurrent polynomials are equivalent to a set of polynomials termed the Dickson polynomials in several variables, which have occurred in many applications in the theory of finite fields~\cite{Glmull}.
\end{example}
\begin{example}
In the case of one variable, that is when $x_{1}=x_{2}=\cdots=x_{k}=x$, and $c_{i}=x^{k-i}$, the recurrent polynomials of order $k$ turn into the $k$-step Lucas polynomials studied and extended recently in~\cite{Oozkann} with the aid of some matrices.
\end{example}
For the remainder of this section, we only consider $P^{(k-1)}_{n}(x_{1},x_{2},\cdots,x_{k})$ instead of $P^{(l)}_{n}(x_{1},x_{2},\cdots,x_{k}),l=0,\ldots,k-2,$ because of its simplicity in calculus and the fact that each polynomial $P^{(l)}_{n}(x_{1},x_{2},\cdots,x_{k}),l=0,\ldots,k-2,$ can be represented in terms of this important recurrent polynomial.\\

For proving some main results, we need the following lemma, which is a consequence of Theorem$1.3$ and Theorem$1.4$ of~\cite{Szriaa}.
\begin{lem}\label{lem1}
Let $c_{1},c_{2},\ldots,c_{k}$ be the elements of a commutative ring with unity. Then
\begin{equation*}
\Bigg(1-\sum_{i=1}^{k}c_{i}X^{i}\Bigg)^{-1}=1+\sum_{n\geq 1}b_{n}X^{n},
\end{equation*}
where
\begin{equation*}
b_{n}=\sum_{i_1+2i_2+\cdots+ki_k=n}\frac{(i_1+i_2+\cdots+i_k)!}{i_1!i_2!\cdots i_k!}c_{1}^{i_1}c_{2}^{i_2}\cdots c_{k}^{i_k}
\end{equation*}
and
\begin{equation*}
b_{n}=(-1)^{n}
\left|\begin{array}{cccccc}
-c_{1}   & -c_{2}        &\cdots  &\cdots   & \cdots  &-c_{n}\\
 1      &     \ddots    &\ddots  &         &        &\vdots      \\
0       &    \ddots     & \ddots &  \ddots &        &\vdots\\
\vdots  &    \ddots     & \ddots &\ddots   &  \ddots      &\vdots\\
\vdots  &               & \ddots &\ddots   & \ddots & -c_{2}\\
0       & \cdots        & \cdots &   0     & 1      &-c_{1}
\end{array}\right|.
\end{equation*}
with $c_{n}=0$ if $n\geq k+1$.
\end{lem}
In the following theorem, we derive a formula of $P^{(k-1)}_{n}(x_{1},x_{2},\cdots,x_{k})$ in terms of $c_{1},c_{2},\ldots,c_{k}$ the coefficients of the polynomial~\eqref{char} which are polynomials in $x_{1},x_{2},\cdots,x_{k}$.
\begin{theorem}\label{thm}
For $n\geq 0$, we have
\begin{equation*}
P^{(k-1)}_{n+k-1}(x_{1},x_{2},\cdots,x_{k})=\sum_{i_1+2i_2+\cdots+ki_k=n}\frac{(i_1+i_2+\cdots+i_k)!}{i_1!i_2!\cdots i_k!}c_{1}^{i_1}c_{2}^{i_2}\cdots c_{k}^{i_k},
\end{equation*}
and
\begin{equation*}
P^{(k-1)}_{n+k-1}(x_{1},x_{2},\cdots,x_{k})=
\left|\begin{array}{cccccc}
c_{1}   & c_{2}        &\cdots  &\cdots   & \cdots  &c_{n}\\
 -1      &     \ddots    &\ddots  &         &        &\vdots      \\
0       &    \ddots     & \ddots &  \ddots &        &\vdots\\
\vdots  &    \ddots     & \ddots &\ddots   &  \ddots      &\vdots\\
\vdots  &               & \ddots &\ddots   & \ddots & c_{2}\\
0       & \cdots        & \cdots &   0     & -1      &c_{1}
\end{array}\right|,
\end{equation*}
with $c_{n}=0$ if $n\geq k+1$. 
\end{theorem}
\begin{proof}
It is a routine to check that
\begin{equation*}
\sum_{n=0}P^{(k-1)}_{n+k-1}(x_{1},x_{2},\cdots,x_{k})X^{n}=\bigg(1-(c_{1}X+\cdots+c_{k}X^{k})\bigg)^{-1}
\end{equation*}
Then the rest follows easily from Lemma~\ref{lem1}.
\end{proof}
Now, we recall an important class of recurrent polynomials of order $k$.
\begin{defn}
The generalized Lucas polynomials in several variables of order $k\geq 2$ are defined by
\begin{equation*}
L_{n+k}(x_{1},x_{2},\cdots,x_{k})=\sum_{i=1}^{k}(-1)^{i+1}x_{i}L_{n+k-i}(x_{1},x_{2},\cdots,x_{k}),
\end{equation*}
with $L_{j}(x_{1},x_{2},\cdots,x_{k})=\delta_{j,k-1}$ for $0\leq j\leq k-1$.
\end{defn}
R. Barakat and E. Baumann~\cite{Rbarak} indicated the great importance of the generalized Lucas polynomials in several variables in many physical problems and asked for obtaining them in a closed-form formula. Here, we provide the solution to this problem in the following interesting theorem.
\begin{theorem}~
For all $n\geq 0$, the generalized Lucas polynomials of order $k\geq 2$ satisfy the following results
\begin{equation*}
L_{n+k-1}(x_{1},x_{2},\ldots,x_{k})=\sum_{i_1+2i_2+\cdots+ki_k=n}\frac{(i_1+i_2+\cdots+i_k)!}{i_1!i_2!\cdots i_k!}(-1)^{k+i_1+i_2+\cdots+i_k}x_{1}^{i_1}x_{2}^{i_2}\cdots x_{k}^{i_k},
\end{equation*}
and
\begin{equation*}
L_{n+k-1}(x_{1},x_{2},\ldots,x_{k})=
\left|\begin{array}{cccccc}
x_{1}   & -x_{2}        &\cdots  &\cdots   & \cdots  &(-1)^{n+1}x_{n}\\
 -1      &     \ddots    &\ddots  &         &        &\vdots      \\
0       &    \ddots     & \ddots &  \ddots &        &\vdots\\
\vdots  &    \ddots     & \ddots &\ddots   &  \ddots      &\vdots\\
\vdots  &               & \ddots &\ddots   & \ddots & -x_{2}\\
0       & \cdots        & \cdots &   0     & -1      &x_{1}
\end{array}\right|_{n\times n}.
\end{equation*}
with $x_{n}=0$ if $n\geq k+1$.
\end{theorem}
\begin{proof}
It is obvious that when $c_{i}=(-1)^{i+1}x_{i}$, we have $L_{n}(x_{1},x_{2},\ldots,x_{k})=P^{(k-1)}_{n}(x_{1},x_{2},\cdots,x_{k})$. The rest can be deduced immediately from Theorem~\ref{thm}.
\end{proof}
\section{Linear recurrence sequences via formal power series}

In this section, we derive explicit formulas of the recurrent polynomial $P^{(k-1)}_{n}(x_{1},x_{2},\cdots,x_{k})$ in terms of the roots of its characteristic polynomial~\eqref{char}. In the sequel, we consider only $R=\mathbb{C}$ the field of complex numbers.
\begin{theorem}
Let $g(X)=\prod_{i=1}^{k}(X-\alpha_{i})=X^{k}-c_{1}X^{k-1}-c_{2}X^{k-2}-\cdots-c_{k}$ be the characteristic polynomial of $P^{(k-1)}_{n}(x_{1},x_{2},\cdots,x_{k})$ having different roots $\alpha_{1}=\alpha_{1}(x_{1},x_{2},\cdots,x_{k}),\alpha_{2}=\alpha_{2}(x_{1},x_{2},\cdots,x_{k}),\ldots,\alpha_{k}=\alpha_{k}(x_{1},x_{2},\cdots,x_{k})$. Then for $n\geq 0$, we have
\begin{equation*}
P^{(k-1)}_{n+k-1}(x_{1},x_{2},\cdots,x_{k})=\sum_{i_1+i_2+\cdots+i_k=n}\alpha_{1}^{i_1}\alpha_{2}^{i_2}\cdots\alpha_{k}^{i_k}.
\end{equation*}
\end{theorem}
\begin{proof}
By algebraic manipulations, we can easily show that
\begin{align*}
\sum_{n\geq 0}P^{(k-1)}_{n+k-1}(x_{1},x_{2},\cdots,x_{k})X^{n}&=\Bigg(1-c_{1}X-\cdots-c_{k}X^{k}\Bigg)^{-1} \\
                    &=\prod_{i=1}^{k}(1-\alpha_{i}X)^{-1}\\
                    &=\prod_{i=1}^{k}(\sum_{n\geq 0}\alpha_{i}^{n}X^{n})\\
                    &=\sum_{n\geq 0}\sum_{i_1+i_2+\cdots+i_k=n}\alpha_{1}^{i_1}\alpha_{2}^{i_2}\cdots\alpha_{k}^{i_k}X^{n}\\
\end{align*}
Consequently
\begin{equation*}
P^{(k-1)}_{n+k-1}(x_{1},x_{2},\cdots,x_{k})=\sum_{i_1+i_2+\cdots+i_k=n}\alpha_{1}^{i_1}\alpha_{2}^{i_2}\cdots\alpha_{k}^{i_k}
\end{equation*}
as claimed in the result.
\end{proof}
The previous result was obtained by C. Levesque in~\cite{Clevesque}. The general case when the characteristic polynomial of $P^{(k-1)}_{n}(x_{1},x_{2},\cdots,x_{k})$ has multiple roots is the following general result.
\begin{theorem}
Let $g(X)=\prod_{i=1}^{s}(X-\alpha_{i})^{m_{i}}=X^{k}-c_{1}X^{k-1}-c_{2}X^{k-2}-\cdots-c_{k}$ be the characteristic polynomial of 
$P^{(k-1)}_{n}(x_{1},x_{2},\cdots,x_{k})$ having different roots $\alpha_{1}=\alpha_{1}(x_{1},x_{2},\cdots,x_{k}),\alpha_{2}=\alpha_{2}(x_{1},x_{2},\cdots,x_{k}),\ldots,\alpha_{s}=\alpha_{s}(x_{1},x_{2},\cdots,x_{k})$ with multiplicities $m_{1},m_{2},\ldots,m_{s}$ respectively. Then for $n\geq 0$, we have the following general formula
\begin{equation*}
P^{(k-1)}_{n+k-1}(x_{1},x_{2},\cdots,x_{k})=\sum_{i_1+i_2+\cdots+i_s=n}\binom{i_{1}+m_{1}-1}{i_{1}}\alpha_{1}^{i_{1}}\cdots \binom{i_{s}+m_{s}-1}{i_{s}}\alpha_{s}^{i_{s}}.
\end{equation*}
\end{theorem}
\begin{proof}
We clearly have
\begin{align*}
\sum_{n\geq 0}P^{(k-1)}_{n+k-1}(x_{1},x_{2},\cdots,x_{k})X^{n}&=(1-c_{1}X-\cdots-c_{k}X^{k})^{-1}\\
                    &=\prod_{i=1}^{s}(1-\alpha_{i}X)^{-m_{i}}
\end{align*}
On the other hand, it is not difficult to show that
\begin{equation*}
(1-\alpha_{i}X)^{-m_{i}}=\sum_{n\geq 0}\binom{n+m_{i}-1}{n}\alpha_{i}^{n}X^{n}
\end{equation*}
this gives
\begin{align*}
\sum_{n\geq 0}P^{(k-1)}_{n+k-1}(x_{1},x_{2},\cdots,x_{k})X^{n}&=\prod_{i=1}^{s}\Bigg(\sum_{n\geq 0}\binom{n+m_{i}-1}{n}\alpha_{i}^{n}X^{n}\Bigg)\\
                    &=\sum_{n\geq 0}\Bigg(\sum_{i_1+i_2+\cdots+i_s=n}\binom{i_{1}+m_{1}-1}{i_{1}}\alpha_{1}^{i_{1}}\cdots \binom{i_{s}+m_{s}-1}{i_{s}}\alpha_{s}^{i_{s}}\Bigg)X^{n}\\
\end{align*}
Therefore
\begin{equation*}
P^{(k-1)}_{n+k-1}(x_{1},x_{2},\cdots,x_{k})=\sum_{i_1+i_2+\cdots+i_s=n}\binom{i_{1}+m_{1}-1}{i_{1}}\alpha_{1}^{i_{1}}\cdots \binom{i_{s}+m_{s}-1}{i_{s}}\alpha_{s}^{i_{s}}
\end{equation*}
Then we reach the desired result.
\end{proof}
In the case of a single root, we have
\begin{corollary}
Let $g(X)=(X-\alpha)^{k}=X^{k}-c_{1}X^{k-1}-c_{2}X^{k-2}-\cdots-c_{k}$ be the characteristic polynomial of $P^{(k-1)}_{n}(x_{1},x_{2},\cdots,x_{k})$ having exactly one root $\alpha=\alpha(x_{1},x_{2},\cdots,x_{k})$. Then for $n\geq 0$, we have
\begin{equation*}
P^{(k-1)}_{n+k-1}(x_{1},x_{2},\cdots,x_{k})=\binom{n+k-1}{n}\alpha^{n}.
\end{equation*}
\end{corollary}
\section{The Fibonacci polynomials of order $2$ and the Lucas polynomials of order $2$}
This section is devoted to recurrent polynomials of order $2$, which are arguably the most important because they have various remarkable properties and a large number of applications in mathematics, computer sciences, physics, and other related topics.
For some references of these applications see~\cite{Harm1,Harm2,Rbarak,Gevere,Tkoshy,Oozkann,Vmow,Mcpea} and the references therein.
\\
\indent
The results of this section can be seen as the generalization and unification of various previously obtained results on sequence polynomials satisfying a second order linear recurrence relation. Taking into account some results of the previous sections, we will split the statements into numerous theorems. Now we are ready to present and prove our results. Let us first start by defining two class of polynomials.
\begin{defn}
Let $q_{1}(x,y)$ and $q_{2}(x,y)$ be two polynomials in $x$ and $y$.
The Fibonacci polynomials of order $2$ are defined by the recurrence relation
\begin{equation*}
F_{n+2}(x,y)=q_{1}(x,y)F_{n+1}(x,y)+q_{2}(x,y)F_{n}(x,y), \,\ n\geq 0,
\end{equation*}
with initial conditions $F_{0}(x,y)=0$ and $F_{1}(x,y)=1$. 
\end{defn}
\begin{defn}
Let $q_{1}(x,y)$ and $q_{2}(x,y)$ be two polynomials in $x$ and $y$.
The Lucas polynomials of order $2$ are defined by the recurrence relation
\begin{equation*}
L_{n+2}(x,y)=q_{1}(x,y)L_{n+1}(x,y)+q_{2}(x,y)L_{n}(x,y), \,\ n\geq 0,
\end{equation*}
with initial conditions $L_{0}(x,y)=2$ and $L_{1}(x,y)=q_{1}(x,y)$.
\end{defn}
\begin{remark}
It is clear that the Fibonacci polynomials and the Lucas polynomials of order $2$ are a natural generalization of the classical Fibonacci polynomials and Lucas polynomials treated recently in~\cite{Oozkann}.
\end{remark}
The following lemma is a useful result which will be required for proving some statements.
\begin{lem}\label{useful lem}
Every recurrent polynomial $P_{n}(x,y)$ of order $2$ with arbitrary initial conditions $P_{0}(x,y),P_{1}(x,y)$ can be written in terms of the Fibonacci polynomials of order $2$ as
\begin{equation*}
P_{n}(x,y)=q_{2}(x,y)P_{0}(x,y)F_{n-1}(x,y)+P_{1}(x,y)F_{n}(x,y), \,\ n\geq 1.
\end{equation*}
\end{lem}
\begin{proof}
Obviously these polynomials verify the same recurrence relation and have the same initial conditions.
\end{proof}
\begin{theorem}\label{thm f5}
For $n\geq 1$, the following identity holds
\begin{equation*}
L_{n}(x,y)=2q_{2}(x,y)F_{n-1}(x,y)+q_{1}(x,y)F_{n}(x,y). 
\end{equation*}
\end{theorem}
\begin{proof}
According to Lemma~\ref{useful lem}, we have
\begin{equation*}
L_{n}(x,y)=q_{2}(x,y)L_{0}(x,y)F_{n-1}(x,y)+L_{1}(x,y)F_{n}(x,y)
\end{equation*}
In fact this gives the desired formula. 
\end{proof}
An easy consequence is the following
\begin{corollary}
For $n\geq 1$, the following identity holds
\begin{equation*}
L_{n}(x)=2F_{n-1}(x)+xF_{n}(x)=F_{n-1}(x)+F_{n+1}(x). 
\end{equation*}
In particular
\begin{equation*}
L_{n}=2F_{n-1}+F_{n}=F_{n-1}+F_{n+1}. 
\end{equation*}
\end{corollary}
\begin{theorem}
For all $n\geq 0$ , we have 
\begin{equation*}
2L_{n+1}(x,y)-q_{1}(x,y)L_{n}(x,y)=\bigg\{q^{2}_{1}(x,y)+4q_{2}(x,y)\bigg\}F_{n}(x,y)
\end{equation*}
\end{theorem}
\begin{proof}
Using the formula~\eqref{eq: g1}, we obtain
$$ \left \{
\begin{array}{rcl}
L_{n}(x,y)&=&L_{0}(x,y)P^{(0)}_{n}(x,y)+L_{1}(x,y)P^{(1)}_{n}(x,y) \vspace*{0.5pc}\\
L_{n+1}(x,y)&=&L_{1}(x,y)P^{(0)}_{n}(x,y)+L_{2}(x,y)P^{(1)}_{n}(x,y) \vspace*{0.5pc}\\
\end{array}
\right.
$$
This entails the matrix identity
$$
\begin{pmatrix}
L_{n}(x,y)\\
L_{n+1}(x,y)
\end{pmatrix}
=\begin{pmatrix}
L_{0}(x,y) & L_{1}(x,y) \\ L_{1}(x,y) & L_{2}(x,y)
\end{pmatrix}
\begin{pmatrix}
P^{(0)}_{n}(x,y)\\
P^{(1)}_{n}(x,y))
\end{pmatrix}
$$
Consequently
$$ \left \{
\begin{array}{rcl}
u(x,y)P^{(0)}_{n}(x,y)&=&L_{2}(x,y)L_{n}(x,y)-L_{1}(x,y)L_{n+1}(x,y) \vspace*{0.5pc}\\
u(x,y)P^{(1)}_{n}(x,y)&=&L_{0}(x,y)L_{n+1}(x,y)-L_{1}(x,y)L_{n}(x,y) \vspace*{0.5pc}\\
\end{array}
\right.
$$
where $u(x,y)=L_{0}(x,y)L_{2}(x,y)-(L_{1}(x,y))^{2}$. From this, we deduce the result.
\end{proof}
We easily deduce the following corollary
\begin{corollary}
For all $n\geq 0$ , we have 
\begin{equation*}
2L_{n+1}(x)-xL_{n}(x)=(x^{2}+4)F_{n}(x)
\end{equation*}
In particular
\begin{equation*}
2L_{n+1}-L_{n}=5F_{n}
\end{equation*}
\end{corollary}
\begin{theorem}\label{easy thm}
For $n\geq 1$ and $p\geq 0$, we have
\begin{equation*}
L_{n+p}(x,y)=q_{2}(x,y)L_{p}(x,y)F_{n-1}(x,y)+L_{p+1}(x,y)F_{n}(x,y).
\end{equation*}
In particular
\begin{equation*}
L_{2n}(x,y)=q_{2}(x,y)L_{n}(x,y)F_{n-1}(x,y)+L_{n+1}(x,y)F_{n}(x,y).
\end{equation*}
\end{theorem}
\begin{proof}
This is result can be easily deduced from Lemma~\ref{useful lem}.
\end{proof}
An immediate consequence is obtained in the following result.
\begin{corollary}
For $n\geq 1$ and $p\geq 0$, we have the following relations
\begin{equation*}
L_{n+p}(x)=L_{p}(x)F_{n-1}(x)+L_{p+1}(x)F_{n}(x),
\end{equation*}
and
\begin{equation*}
L_{2n}(x)=L_{n}(x)F_{n-1}(x)+L_{n+1}(x)F_{n}(x).
\end{equation*}
In particular
\begin{equation*}
L_{n+p}=L_{p}F_{n-1}+L_{p+1}F_{n},
\end{equation*}
and
\begin{equation*}
L_{2n}=L_{n}F_{n-1}+L_{n+1}F_{n}.
\end{equation*}
\end{corollary}
In the following theorem we express $L_{2n}(x,y)$ only in terms of the Fibonacci polynomials of order $2$
\begin{theorem}
For any $n\geq 1$, we have the following identity
\begin{align*}
L_{2n}(x,y)&=F_{n+1}^{2}(x,y)+2q_{2}(x,y)F_{n}^{2}(x,y)+q_{2}^{2}(x,y)F_{n-1}^{2}(x,y)
\end{align*}
\end{theorem}
\begin{proof}
Using Theorem~\ref{easy thm}, we have
\begin{equation*}
L_{2n}(x,y)=q_{2}(x,y)L_{n}(x,y)F_{n-1}(x,y)+L_{n+1}(x,y)F_{n}(x,y)
\end{equation*}
and by Theorem~\ref{thm f5}, we have
$$ \left \{
\begin{array}{rcl}
L_{n}(x,y)&=&2q_{2}(x,y)F_{n-1}(x,y)+q_{1}(x,y)F_{n}(x,y) \vspace*{0.5pc}\\
L_{n+1}(x,y)&=&q_{1}(x,y)q_{2}(x,y)F_{n-1}(x,y)+\bigg\{2q_{2}(x,y)+q_{1}^{2}(x,y)\bigg\}F_{n}(x,y) \vspace*{0.5pc}\\
\end{array}
\right.
$$
It follows that
\begin{align*}
L_{2n}(x,y)&=q_{2}(x,y)\bigg\{2q_{2}(x,y)F_{n-1}(x,y)+q_{1}(x,y)F_{n}(x,y)\bigg\}F_{n-1}(x,y)\\                                                                                             
                   &+\bigg\{q_{1}(x,y)q_{2}(x,y)F_{n-1}(x,y)+\bigg(2q_{2}(x,y)+q_{1}^{2}(x,y)\bigg)F_{n}(x,y)\bigg\}F_{n}(x,y)\\
                  &=\bigg(2q_{2}(x,y)+q_{1}^{2}(x,y)\bigg)F_{n}^{2}(x,y)+2q_{2}^{2}(x,y)F_{n-1}^{2}(x,y)\\
                  &+2q_{1}(x,y)q_{2}(x,y)F_{n-1}(x,y)F_{n}(x,y)\\
                  &=q_{1}^{2}(x,y)F_{n}^{2}(x,y)+q_{2}^{2}(x,y)F_{n-1}^{2}(x,y)\\&+2q_{1}(x,y)q_{2}(x,y)F_{n-1}(x,y)F_{n}(x,y)\\
                  &+2q_{2}(x,y)F_{n}^{2}(x,y)+q_{2}^{2}(x,y)F_{n-1}^{2}(x,y)\\
                  &=\bigg(q_{1}(x,y)F_{n}(x,y)+q_{2}(x,y)F_{n-1}(x,y)\bigg)^{2}+2q_{2}(x,y)F_{n}^{2}(x,y)\\&+q_{2}^{2}(x,y)F_{n-1}^{2}(x,y)\\
                  &=F_{n+1}^{2}(x,y)+2q_{2}(x,y)F_{n}^{2}(x,y)+q_{2}^{2}(x,y)F_{n-1}^{2}(x,y)
\end{align*}
as claimed.
\end{proof}
The last theorem can be used to establish further identities on the Classical Fibonacci and Lucas numbers/polynomials, for example we state the following corollary.
\begin{corollary}
For any $n\geq 1$, we have the following identity
\begin{equation*}
L_{2n}(x)=F_{n+1}^{2}(x)+2F_{n}^{2}(x)+F_{n-1}^{2}(x)
\end{equation*}
In particular
\begin{equation*}
L_{2n}=F_{n+1}^{2}+2F_{n}^{2}+F_{n-1}^{2}
\end{equation*}
\end{corollary}
In the following theorem, we show that the Fibonacci polynomials of order $2$ have a beautiful determinantal expression.
\begin{theorem}\label{thm f6}
For any two polynomials $q_{1}(x,y)$ and $q_{2}(x,y)$ and for every $n\geq 1$, the following determinantal identity holds
\begin{equation*}
F_{n+1}(x,y)=
\left|\begin{array}{cccccc}
q_{1}(x,y)   & q_{2}(x,y) & 0 &\cdots   & \cdots  &0\\
 -1      &     \ddots    &\ddots  & \ddots      &       &\vdots      \\
0       &    \ddots     & \ddots &  \ddots & \ddots       &\vdots\\
\vdots  &    \ddots     & \ddots &\ddots   &  \ddots      &0\\
\vdots  &               & \ddots &\ddots   & \ddots & q_{2}(x,y)\\
0       & \cdots        & \cdots &   0     & -1      & q_{1}(x,y)
\end{array}\right|_{n\times n}
\end{equation*}

\end{theorem}
\begin{proof}
Using standard algebraic techniques, we can write
\begin{equation*}
\sum_{n\geq 0}F_{n+1}(x,y)X^{n}=(1-q_{1}(x,y)X-q_{2}(x,y)X^{2})^{-1}
\end{equation*}
and the result follows easily by Lemma~\ref{lem1}.
\end{proof}
From this theorem we easily deduce the following.
\begin{corollary}
For all $n\geq 0$, the classical Fibonacci polynomials $F_{n}(x)$ can be represented in the following determinantal expression
\begin{equation*}
F_{n+1}(x)=
\left|\begin{array}{cccccc}
x   & 1 & 0 &\cdots   & \cdots  &0\\
 -1      &     \ddots    &\ddots  & \ddots      &       &\vdots      \\
0       &    \ddots     & \ddots &  \ddots & \ddots       &\vdots\\
\vdots  &    \ddots     & \ddots &\ddots   &  \ddots      &0\\
\vdots  &               & \ddots &\ddots   & \ddots & 1\\
0       & \cdots        & \cdots &   0     & -1      & x
\end{array}\right|_{n\times n}
\end{equation*}
In particular, we can express the Fibonacci numbers $F_{n}$ in terms of a tridiagonal determinant by
\begin{equation*}
F_{n+1}=
\left|\begin{array}{cccccc}
1   & 1 & 0 &\cdots   & \cdots  &0\\
 -1      &     \ddots    &\ddots  & \ddots      &       &\vdots      \\
0       &    \ddots     & \ddots &  \ddots & \ddots       &\vdots\\
\vdots  &    \ddots     & \ddots &\ddots   &  \ddots      &0\\
\vdots  &               & \ddots &\ddots   & \ddots & 1\\
0       & \cdots        & \cdots &   0     & -1      & 1
\end{array}\right|_{n\times n}
\end{equation*}
for $n\geq 0$.
\end{corollary}
\begin{theorem}\label{thm:ssss}
For all $n\geq 0$, we have a relation between the Fibonacci polynomials of order $2$ and the Lucas polynomials of order $2$ as follows,

\begin{equation*}
L_{n}(x,y)=2F_{n+1}(x,y)-q_{1}(x,y)F_{n}(x,y)
\end{equation*}
\end{theorem}
\begin{proof}
An easy computation shows that
\begin{equation*}
\sum_{n\geq 0}L_{n}(x,y)X^{n}=\bigg(2-q_{1}(x,y)X\bigg)\bigg(1-q_{1}(x,y)X-q_{2}(x,y)X^{2}\bigg)^{-1}
\end{equation*}
This implies
\begin{equation*}
\sum_{n\geq 0}L_{n}(x,y)X^{n}=\bigg(2-q_{1}(x,y)X\bigg)\sum_{n\geq 0}F_{n+1}(x,y)X^{n}
\end{equation*}
Thus
\begin{equation*}
\sum_{n\geq 0}L_{n}(x,y)X^{n}=2+\sum_{n\geq 1}(2F_{n+1}(x,y)-q_{1}(x,y)F_{n}(x,y))X^{n}
\end{equation*}
Consequently, the identification of coefficients gives
\begin{equation*}
L_{n}(x,y)=2F_{n+1}(x,y)-q_{1}(x,y)F_{n}(x,y) ,\,\,\ n\geq 0.
\end{equation*}
this proves the result.
\end{proof}
\begin{corollary}
For $n\geq 0$, we have
\begin{equation*}
L_{n}(x)=2F_{n+1}(x)-xF_{n}(x)
\end{equation*}
In particular
\begin{equation*}
L_{n}=2F_{n+1}-F_{n}
\end{equation*}
\end{corollary}
In the following result we develop an interesting explicit formula for $F_{n}(x,y)$.
\begin{theorem}\label{thm f7}
For any two polynomials $q_{1}(x,y)$ and $q_{2}(x,y)$ the Fibonacci polynomials of order $2$ may be expressed in terms of powers of $q_{1}(x,y)$ and $q_{2}(x,y)$ as follows:
\begin{equation*}
F_{n+1}(x,y)=\sum_{i=0}^{[\frac{n}{2}]}\binom{n-i}{i}(q_{1}(x,y))^{n-2i}(q_{2}(x,y))^{i}
\end{equation*}
for $n\geq0$.
\end{theorem}
\begin{proof}
Since
\begin{equation*}
F_{n}(x,y)=P_{n}^{(1)}(x,y)
\end{equation*}
Then using Theorem~\ref{thm}, we have
\begin{align*}
F_{n+1}(x,y)&=\sum_{i_1+2i_2=n}\frac{(i_1+i_2)!}{i_1!i_2!}(q_{1}(x,y))^{i_1}(q_{2}(x,y))^{i_2}\\
                  &=\sum_{i=0}^{[\frac{n}{2}]}\binom{n-i}{i}(q_{1}(x,y))^{n-2i}(q_{2}(x,y))^{i}
\end{align*}
Therefore the formula is proved.
\end{proof}
As a consequence, we obtain explicit formulas for $F_{n}(x)$ and $F_{n}$.
\begin{corollary}
For any $n\geq0$, the classical Fibonacci polynomials may be expressed as follows:
\begin{equation*}
F_{n+1}(x)=\sum_{i=0}^{[\frac{n}{2}]}\binom{n-i}{i}x^{n-2i},
\end{equation*}
and the Fibonacci numbers can be written as:
\begin{equation*}
F_{n+1}=\sum_{i=0}^{[\frac{n}{2}]}\binom{n-i}{i}.
\end{equation*}
\end{corollary}
The following theorem provides an explicit formula for $L_{n}(x,y)$.
\begin{theorem}\label{thm:lucas}
For any two polynomials $q_{1}(x,y)$ and $q_{2}(x,y)$ the Lucas polynomials of order $2$ may be expressed in terms of powers of $q_{1}(x,y)$ and $q_{2}(x,y)$ as follows:
\begin{equation*}
L_{n}(x,y)=\sum_{i=0}^{[\frac{n}{2}]}\frac{n}{n-i}\binom{n-i}{i}(q_{1}(x,y))^{n-2i}(q_{2}(x,y))^{i}
\end{equation*}
for $n\geq1$.
\end{theorem}
\begin{proof}
Theorem\ref{thm:ssss} tells us
\begin{equation*}
L_{n}(x,y)=2F_{n+1}(x,y)-q_{1}(x,y)F_{n}(x,y)
\end{equation*}

And by Theorem~\ref{thm f7}, we have
\begin{align*}
L_{n}(x,y)&= 2\sum_{i=0}^{[\frac{n}{2}]}\binom{n-i}{i}(q_{1}(x,y))^{n-2i}(q_{2}(x,y))^{i}\\
                 &-q_{1}(x,y)\sum_{i=0}^{[\frac{n-1}{2}]}\binom{n-i-1}{i}(q_{1}(x,y))^{n-2i-1}(q_{2}(x,y))^{i}\\
                 &=\sum_{i=0}^{[\frac{n}{2}]}\bigg\{(2q_{1}(x,y)\binom{n-i}{i}-q_{1}(x,y)\binom{n-i-1}{i}\bigg\}(q_{1}(x,y))^{n-2i-1}(q_{2}(x,y))^{i}\\
                 &=\sum_{i=0}^{[\frac{n}{2}]}\bigg\{2\binom{n-i}{i}-\binom{n-i-1}{i}\bigg\}(q_{1}(x,y))^{n-2i}(q_{2}(x,y))^{i}\\
                 &=\sum_{i=0}^{[\frac{n}{2}]}\frac{n}{n-i}\binom{n-i}{i}(q_{1}(x,y))^{n-2i}(q_{2}(x,y))^{i} 
\end{align*}
This completes the proof.
\end{proof}
Correspondingly, we can produce explicit formulas for the classical Lucas numbers/polynomials.
\begin{corollary}
For any $n\geq0$, the classical Lucas polynomials may be expressed as follows:
\begin{equation*}
L_{n}(x)=\sum_{i=0}^{[\frac{n}{2}]}\frac{n}{n-i}\binom{n-i}{i}x^{n-2i}
\end{equation*}
As a consequence, the Lucas numbers can be written as:
\begin{equation*}
L_{n}=\sum_{i=0}^{[\frac{n}{2}]}\frac{n}{n-i}\binom{n-i}{i}
\end{equation*}
\end{corollary}
\begin{theorem}\label{thm: fl}
For any two polynomials $q_{1}(x,y)$ and $q_{2}(x,y)$. Let $$\alpha(x,y)=\frac{1}{2}\Big(q_{1}(x,y)+\sqrt{(q_{1}(x,y))^{2}+4q_{2}(x,y)}\Big)$$ and $$\beta(x,y)=\frac{1}{2}\Big(q_{1}(x,y)-\sqrt{(q_{1}(x,y))^{2}+4q_{2}(x,y)}\Big)$$ be the roots of the polynomial $P(t)=t^{2}-q_{1}(x,y)t-q_{2}(x,y)$, then
\begin{enumerate}
\item[1)] For every $n\geq 0$, we have
\begin{equation*}
F_{n}(x,y)=\frac{\alpha^{n}(x,y)-\beta^{n}(x,y)}{\alpha(x,y)-\beta(x,y)}
\end{equation*}
\item[2)] For every $n\geq 0$, we have
\begin{equation*}
L_{n}(x,y)=\alpha^{n}(x,y)+\beta^{n}(x,y) 
\end{equation*}
\item[3)] For every $n\geq 0$ and $m\geq 0$, we have
\begin{align*}
\bigg(L_{n}(x,y)+\sqrt{(q_{1}(x,y))^{2}+4q_{2}(x,y)}F_{n}(x,y)\bigg)^{m}&=\\2^{m-1}\bigg(L_{nm}(x,y)+\sqrt{(q_{1}(x,y))^{2}+4q_{2}(x,y)}F_{nm}(x,y)\bigg)
\end{align*}
\item[4)] For every $n\geq 0$ and $m\geq 0$, we have
\begin{align*}
\bigg(L_{n}(x,y)-\sqrt{(q_{1}(x,y))^{2}+4q_{2}(x,y)}F_{n}(x,y)\bigg)^{m}&=\\2^{m-1}\bigg(L_{nm}(x,y)-\sqrt{(q_{1}(x,y))^{2}+4q_{2}(x,y)}F_{nm}(x,y)\bigg)
\end{align*}
\item[5)] For $n\geq 0$ and $m\geq 0$, we have
\begin{align*}
\bigg(L_{n}(x,y)+\sqrt{(q_{1}(x,y))^{2}+4q_{2}(x,y)}F_{n}(x,y)\bigg)^{m}\\+\bigg(L_{n}(x,y)-\sqrt{(q_{1}(x,y))^{2}+4q_{2}(x,y)}F_{n}(x,y)\bigg)^{m}&=2^{m}L_{nm}(x,y)
\end{align*}
\end{enumerate}
\end{theorem}
\begin{proof}
Using~\eqref{eq: g1}, we can easily obtain
$$ \left \{
\begin{array}{rcl}
\alpha^{n}(x,y)&=&q_{2}(x,y)F_{n-1}(x,y)+\alpha(x,y)F_{n}(x,y) \vspace*{0.5pc}\\
\beta^{n}(x,y)&=&q_{2}(x,y)F_{n-1}(x,y)+\beta(x,y)F_{n}(x,y) \vspace*{0.5pc}\\
\end{array}
\right.
$$
This implies
\begin{equation*}
F_{n}(x,y)=\frac{\alpha^{n}(x,y)-\beta^{n}(x,y)}{\alpha(x,y)-\beta(x,y)}
\end{equation*}
On the other hand, we have by Theorem~\ref{thm f5}
\begin{equation*}
L_{n}(x,y)=2q_{2}(x,y)F_{n-1}(x,y)+q_{1}(x,y)F_{n}(x,y)
\end{equation*}
Using the last identity and the fact that $\alpha(x,y)\beta(x,y)=-q_{2}(x,y)$, we have
\begin{align*}
L_{n}(x,y)&=2q_{2}(x,y)\frac{\alpha^{n-1}(x,y)-\beta^{n-1}(x,y)}{\alpha(x,y)-\beta(x,y)}+q_{1}(x,y)\frac{\alpha^{n}(x,y)-\beta^{n}(x,y)}{\alpha(x,y)-\beta(x,y)}\\
                  &=\frac{q_{1}(x,y)\alpha(x,y)+2q_{2}(x,y)}{\alpha(x,y)-\beta(x,y)}\alpha^{n-1}(x,y)\\&-\frac{q_{1}(x,y)\beta(x,y)+2q_{2}(x,y)}{\alpha(x,y)-\beta(x,y)}\beta^{n-1}(x,y)\\
                  &=\frac{q_{1}(x,y)-2\beta(x,y)}{\alpha(x,y)-\beta(x,y)}\alpha^{n}(x,y)+\frac{2\alpha(x,y)-q_{1}(x,y)}{\alpha(x,y)-\beta(x,y)}\beta^{n}(x,y)
\end{align*}
Since $\alpha(x,y)+\beta(x,y)=q_{1}(x,y)$, then
\begin{equation*}
L_{n}(x,y)=\alpha^{n}(x,y)+\beta^{n}(x,y)
\end{equation*}
Now, it is clear that
\begin{align*}
\bigg(L_{n}(x,y)+\sqrt{(q_{1}(x,y))^{2}+4q_{2}(x,y)}F_{n}(x,y)\bigg)^{m}&=\bigg(\alpha^{n}(x,y)+\beta^{n}(x,y)+\alpha^{n}(x,y)-\beta^{n}(x,y)\bigg)^{m}\\
               &=2^{m-1}(2\alpha^{nm}(x,y)+\beta^{nm}(x,y)-\beta^{nm}(x,y))\\
               &=2^{m-1}(L_{nm}(x,y)+\sqrt{(q_{1}(x,y))^{2}+4q_{2}(x,y)}F_{nm}(x,y))
\end{align*}
The rest of the proof is similar.
\end{proof}

Using Theorem~\ref{thm: fl}, we can immediately obtain the following interesting corollary
\begin{corollary}
For every $n\geq 0$ and $m\geq 0$, the classical Fibonacci polynomials $F_{n}(x)$ and the classical Lucas polynomials $L_{n}(x)$ satisfy
\begin{itemize}
\item[1)] 
\begin{align*}
\bigg(L_{n}(x)+\sqrt{x^{2}+4}F_{n}(x)\bigg)^{m}&=2^{m-1}\bigg(L_{nm}(x)+\sqrt{x^{2}+4}F_{nm}(x)\bigg)
\end{align*}
\item[2)] 
\begin{align*}
\bigg(L_{n}(x)-\sqrt{x^{2}+4}F_{n}(x)\bigg)^{m}&=2^{m-1}\bigg(L_{nm}(x)-\sqrt{x^{2}+4}F_{nm}(x)\bigg)
\end{align*}
\item[3)]
\begin{align*}
\bigg(L_{n}(x)+\sqrt{x^{2}+4}F_{n}(x)\bigg)^{m}+\bigg(L_{n}(x)-\sqrt{x^{2}+4}F_{n}(x)\bigg)^{m}&=2^{m}L_{nm}(x)
\end{align*}
\end{itemize}
\end{corollary}
As a direct consequence of the previous corollary, we obtain the following interesting identities
\begin{corollary}
For every $n\geq 0$ and $m\geq 0$, the Fibonacci numbers $F_{n}$ and the Lucas numbers $L_{n}$ satisfy
\begin{itemize}
\item[1)]
\begin{align*}
(L_{n}+\sqrt{5}F_{n})^{m}&=2^{m-1}(L_{nm}+\sqrt{5}F_{nm})
\end{align*}
\item[2)]
\begin{align*}
(L_{n}-\sqrt{5}F_{n})^{m}&=2^{m-1}(L_{nm}-\sqrt{5}F_{nm})
\end{align*}
\item[3)]
\begin{align*}
(L_{n}+\sqrt{5}F_{n})^{m}+(L_{n}-\sqrt{5}F_{n})^{m}&=2^{m}L_{nm}
\end{align*}
\end{itemize}
\end{corollary}
\section{Consequences on the Dickson polynomials of the first and the second kind}
In this section we investigate an important family of sequence polynomials called Dickson polynomials. Notably, we apply the results of the last section to produce a number of interesting statements. It is worth noting that Dickson polynomials have numerous algebraic and number-theoretic properties. It has turned out that these polynomials have many applications in cryptography, pseudoprimality testing, coding theory, combinatorial design theory, and related topics. Dickson polynomials are of great importance in the theory of permutation polynomials over finite fields. Particularly, these intersting polynomials serve as fundadental examples of integral polynomials that induce permutations for infinitely many primes. We refer the intersted reader to the very nice book on Dickson polynomials~\cite{Glmull}, that presents a comprehensive collection of results of these polynomials and provides a number of applications. Dickson polynomials as one of the timeless mathematical topics attracted the attention of many researchers since their appearance. Actually, Dickson polynomials are an essential research topic open to further progress.\\
\indent
The Dickson polynomials of the first kind $D_{n}(x,a)$ can be generated by
\begin{equation*}
D_{0}(x,a)=2,\,\ D_{1}(x,a)=x, \,\,\ D_{n+2}(x,a)=xD_{n+1}(x,a)-aD_{n}(x,a),\,\ n\geq 0,
\end{equation*}
The first few examples of the Dickson polynomials of the first kind are
\begin{align*}
D_{1}(x,a)&= x \\
D_{2}(x,a)&= x^{2}-2a \\
D_{3}(x,a)&= x^{3}-3xa \\
D_{4}(x,a)&= x^{4}-4x^{2}a+2a^{2} \\
D_{5}(x,a)&= x^{5}-5x^{3}a+5xa^{2} \\
\end{align*}
The Dickson polynomials of the second kind $E_{n}(x,a)$ can be defined by
\begin{equation*}
E_{0}(x,a)=1,\,\ E_{1}(x,a)=x, \,\,\ E_{n+2}(x,a)=xE_{n+1}(x,a)-aE_{n}(x,a),\,\ n\geq 0,
\end{equation*}
These polynomials have not studied much in the literature. The first few examples of the Dickson polynomials of the second kind are
\begin{align*}
E_{1}(x,a)&= x \\
E_{2}(x,a)&= x^{2}-a \\
E_{3}(x,a)&= x^{3}-2xa \\
E_{4}(x,a)&= x^{4}-3x^{2}a+a^{2} \\
E_{5}(x,a)&= x^{5}-4x^{3}a+3xa^{2} \\
\end{align*}
It is clear that the Dickson polynomials of the first kind are particular cases of the Lucas polynomials of order $2$, that is when $q_{1}(x,a)=x,q_{2}(x,a)=-a$. In the rest of the paper, we assume that $q_{1}(x,a)=x,q_{2}(x,a)=-a$.\\
\indent
In the following result, we describe a relationship between Dickson polynomials of the second kind and the Fibonacci polynomials of order $2$.
\begin{theorem}
The Dickson polynomials of the second kind satisfy the following
\begin{equation*}
E_{n}(x,a)=F_{n+1}(x,a)
\end{equation*}
for every $n\geq 0$.
\end{theorem}
\begin{proof}
By~\eqref{eq: g1}, it is clear that for every $n\geq 1$ we have
\begin{align*}
E_{n}(x,a)&=P^{(0)}_{n}(x,a)+xP^{(1)}_{n}(x,a)\\
          &=-aP^{(1)}_{n-1}(x,a)+xP^{(1)}_{n}(x,a)\\
          &=-aF_{n-1}(x,a)+xF_{n}(x,a)\\
          &=F_{n+1}(x,a)
\end{align*}
which gives the desired formula.
\end{proof}
In determinant form the Dickson polynomials of the second kind are given by  
\begin{theorem}
For every $n\geq 1$
\begin{equation*}
E_{n}(x,a)=
\left|\begin{array}{cccccc}
x   & a        & 0  &\cdots   & \cdots  &0\\
 1      &     \ddots    &\ddots  &   \ddots      &        &\vdots      \\
0       &    \ddots     & \ddots &  \ddots & \ddots       &\vdots\\
\vdots  &    \ddots     & \ddots &\ddots   &  \ddots      &0\\
\vdots  &               & \ddots &\ddots   & \ddots & a\\
0       & \cdots        & \cdots &   0     & 1      &x
\end{array}\right|_{n\times n}
\end{equation*}
\end{theorem}
\begin{proof}
According to the above Theorem, we have $ E_{n}(x,a)=F_{n+1}(x,a), n\geq 1$, and by using Theorem~\ref{thm f6}, we deduce that
\begin{equation*}
E_{n}(x,a)=
\left|\begin{array}{cccccc}
x   & -a        & 0  &\cdots   & \cdots  &0\\
 -1      &     \ddots    &\ddots  &   \ddots      &        &\vdots      \\
0       &    \ddots     & \ddots &  \ddots & \ddots       &\vdots\\
\vdots  &    \ddots     & \ddots &\ddots   &  \ddots      &0\\
\vdots  &               & \ddots &\ddots   & \ddots & -a\\
0       & \cdots        & \cdots &   0     & -1      &x
\end{array}\right|_{n\times n}
\end{equation*}
On the other hand, it is not hard to prove that
\begin{equation*}
\left|\begin{array}{cccccc}
b   & c        & 0  &\cdots   & \cdots  &0\\
 a      &     \ddots    &\ddots  &   \ddots      &        &\vdots      \\
0       &    \ddots     & \ddots &  \ddots & \ddots       &\vdots\\
\vdots  &    \ddots     & \ddots &\ddots   &  \ddots      &0\\
\vdots  &               & \ddots &\ddots   & \ddots & c\\
0       & \cdots        & \cdots &   0     & a      &b
\end{array}\right|_{n\times n}
=\left\{\begin{array}{lll}
\displaystyle\frac{(b+\sqrt{b^{2}-4ac})^{n+1}-(b-\sqrt{b^{2}-4ac})^{n+1}}{2^{n+1}\sqrt{b^{2}-4ac}},\quad b^{2}\neq4ac,\\
(n+1)\bigg(\dfrac{b}{2}\bigg)^{n},\,\,\,\ b^{2}=4ac.\\
\end{array}\right.
\end{equation*}
Consequently, we obtain

\begin{equation*}
\left|\begin{array}{cccccc}
x   & -a        & 0  &\cdots   & \cdots  &0\\
 -1      &     \ddots    &\ddots  &   \ddots      &        &\vdots      \\
0       &    \ddots     & \ddots &  \ddots & \ddots       &\vdots\\
\vdots  &    \ddots     & \ddots &\ddots   &  \ddots      &0\\
\vdots  &               & \ddots &\ddots   & \ddots & -a\\
0       & \cdots        & \cdots &   0     & -1      &x
\end{array}\right|=
\left|\begin{array}{cccccc}
x   & a        & 0  &\cdots   & \cdots  &0\\
 1      &     \ddots    &\ddots  &   \ddots      &        &\vdots      \\
0       &    \ddots     & \ddots &  \ddots & \ddots       &\vdots\\
\vdots  &    \ddots     & \ddots &\ddots   &  \ddots      &0\\
\vdots  &               & \ddots &\ddots   & \ddots & a\\
0       & \cdots        & \cdots &   0     & 1      &x
\end{array}\right|
\end{equation*}
This proves the desired determinantal formula.
\end{proof}
For $a=1$ we get the following result stated in~\cite{Glmull} and~\cite{Jjsylve} without proof
\begin{corollary}
For every $n\geq 1$
\begin{equation*}
E_{n}(x,1)=
\left|\begin{array}{cccccc}
x   & 1       &\cdots  &\cdots   & \cdots  &0\\
 1      &     \ddots    &\ddots  &         &        &\vdots      \\
0       &    \ddots     & \ddots &  \ddots &        &\vdots\\
\vdots  &    \ddots     & \ddots &\ddots   &  \ddots      &\vdots\\
\vdots  &               & \ddots &\ddots   & \ddots & 1\\
0       & \cdots        & \cdots &   0     & 1      &x
\end{array}\right|
\end{equation*}
\end{corollary}
\begin{remark}
Note that the polynomial $E_{n}(x,1)$ appears as the numerator and the denominator of the approximant to some continued fractions, for more details see~\cite[p.15]{Glmull}.
\end{remark}
In the following result, we provide explicit formulas for the Dickson polynomials of the first and the second kind.
\begin{theorem}
For $n\geq 0$ the following identities hold true
\begin{equation*}
D_{n}(x,a)=\sum_{i=0}^{[\frac{n}{2}]}\frac{n}{n-i}\binom{n-i}{i}x^{n-2i}(-a)^{i},
\end{equation*}
and
\begin{equation*}
E_{n}(x,a)=\sum_{i=0}^{[\frac{n}{2}]}\binom{n-i}{i}x^{n-2i}(-a)^{i}.
\end{equation*}
\end{theorem}
\begin{proof}
It is clear that $D_{n}(x,a)=L_{n}(x,a)$ when $q_{1}(x,a)=x$ and $q_{2}(x,a)=-a$, so the first identity follows form Theorem~\ref{thm:lucas}. The second identity can be proved easily with the aid of Theorem~\ref{thm f7}.
\end{proof}
Taking into account Theorem~\ref{thm: fl}, we can easily obtain the following interesting identities
\begin{theorem}
\begin{enumerate}
\item[1)] For every $n\geq 1$ and $m\geq 1$, we have
\begin{equation*}
\bigg(D_{n}(x,a)+\sqrt{x^{2}-4a}E_{n-1}(x,a)\bigg)^{m}=2^{m-1}\bigg(D_{nm}(x,a)+\sqrt{x^{2}-4a}E_{nm-1}(x,a)\bigg)
\end{equation*}
\item[2)] For every $n\geq 1$ and $m\geq 1$, we have
\begin{equation*}
\bigg(D_{n}(x,a)-\sqrt{x^{2}-4a}E_{n-1}(x,a)\bigg)^{m}=2^{m-1}\bigg(D_{nm}(x,a)-\sqrt{x^{2}-4a}E_{nm-1}(x,a)\bigg)
\end{equation*}
\item[3)] For every $n\geq 1$ and $m\geq 0$, we have
\begin{equation*}
\bigg(D_{n}(x,a)+\sqrt{x^{2}-4a}E_{n-1}(x,a)\bigg)^{m}+\bigg(D_{n}(x,a)-\sqrt{x^{2}-4a}E_{n-1}(x,a)\bigg)^{m}=2^{m}D_{nm}(x,a).
\end{equation*}
\end{enumerate}
\end{theorem}
Finally, making use of the results of the last section we obtain the following theorem, which provides some relationships between Dickson polynomials of the first and the second kind.
\begin{theorem}
The following identities hold true:

\begin{enumerate}
\item[a)]
\begin{align*}
D_{n}(x,a)=xE_{n-1}(x,a)-2aE_{n-2}(x,a), \,\,n\geq 2,  
\end{align*}
\item[b)]
\begin{align*}
2D_{n+1}(x,a)-xD_{n}(x,a)=(x^{2}-4a)E_{n-1}(x,a), \,\,n\geq 1,
\end{align*}
\item[c)]
\begin{align*}
D_{n+p}(x,a)=D_{p+1}(x,a)E_{n-1}(x,a)-aD_{p}(x,a)E_{n-2}(x,a), \,\,n\geq 2,p\geq 0,
\end{align*}
\item[d)]
\begin{align*}
D_{2n}(x,a)=D_{n+1}(x,a)E_{n-1}(x,a)-aD_{n}(x,a)E_{n-2}(x,a), \,\,n\geq 2,
\end{align*}
\item[d)]
\begin{align*}
D_{2n}(x,a)=E_{n+1}^{2}(x,a)-2aE_{n-1}^{2}(x,a)+a^{2}E_{n-2}^{2}(x,a), \,\,n\geq 2,
\end{align*}
\item[e)]
\begin{align*}
D_{n}(x,a)=2E_{n}(x,a)-xE_{n-1}(x,a), \,\,n\geq 1.
\end{align*}
\end{enumerate}
\end{theorem}
\section*{Declarations}
\subsection*{Ethical Approval}  
Not applicable.
\subsection*{Competing interests} 
No potential conflict of interest was reported by the authors.
\subsection*{Authors' contributions} 
The authors contributed equally.
\subsection*{Funding} 
This research received no funding.
\subsection*{Availability of data and materials}
Not applicable.

\end{document}